\documentclass{amsart}
\usepackage{graphicx}
\usepackage{amssymb}
\usepackage{graphicx}
\usepackage{amsfonts}
\usepackage{amsmath}
\usepackage{amsthm}
\usepackage{esint}
\usepackage{mathabx}
\usepackage{mathtools}
\usepackage{nccmath}
\usepackage[colorlinks]{hyperref}

\newtheorem{theorem}{Theorem}[section]
\newtheorem{lemma}{Lemma}[section]
\newtheorem{rmk}{Remark}[section]

\newtheorem{definition}{Definition}[section]
\newcommand{\bigO}{\mathcal{O}}
\newcommand{\e}{\epsilon}
\newcommand{\B}{\mathcal{B}}
\newcommand{\F}{\mathcal{F}}
\newcommand{\T}{\mathcal{T}}

\begin{document}

\title[Global proof of the parabolic Harnack inequality]{A global proof of the weak Harnack inequality for parabolic equations in non-divergence form}

\author{Aranya Sen}
\address{Department of Mathematics, UC Irvine}
\email{\tt aranyas1@uci.edu}


\begin{abstract}
The weak Harnack inequality is a fundamental result in the theory of fully nonlinear parabolic partial differential equations. It has several important consequences such as the Hölder continuity of derivatives of solutions to certain fully nonlinear parabolic equations. Classical proofs of the weak Harnack inequality rely on decay of measure estimates, which require delicate localization and covering arguments. We give a new proof based on a detailed study of particular envelopes of solutions to parabolic partial differential equations. This allows us to completely bypass the need of covering arguments. We also prove a $W^{2,\epsilon}$ estimate with an improved dependence of $\epsilon$ on the ellipticity ratio compared to previous results.

\end{abstract}
\maketitle



\section{Introduction}\label{intro}

In this paper we study viscosity supersolutions of
\begin{equation}\label{main_equation}
    u_t-\mathcal{M}^{-}_{\lambda,\Lambda}(D^2u)+C |\nabla u| + C|u|  \geq 0 
\end{equation}
on $Q_2 = B_2(0)\times (-4,0]$ for any $C>0$. $\mathcal{M}^{-}_{\lambda,\Lambda}$ and $\mathcal{M}^{+}_{\lambda,\Lambda}$ are Pucci extremal operators defined in Section \ref{Preliminaries}. These equations appear in the study of solutions of parabolic partial differential equations of the form
\begin{equation}\label{main_pde}
    u_t - F(D^2u,\nabla u, u,x,t) = 0.
\end{equation}
Here $F: Sym(n) \times \mathbb{R}^n\times \mathbb{R} \times Q_2\to \mathbb{R}$ is a function with the following properties,
\begin{enumerate}
    \item (uniform ellipticity) For constants $0<\lambda \leq \Lambda$,
            \begin{equation}
                \lambda \|N\| \leq F(M+N,p,z,x,t) - F(M,p,z,x,t) \leq \Lambda \|N\|
            \end{equation}
            for all $M,N\in Sym(n)$, $N \geq 0$, $(p,z) \in \mathbb{R}^n\times \mathbb{R}$ and $(x,t)\in Q_2$.
    \item ($0$ is a solution) The constant function $u=0$ is a solution to the equation \eqref{main_pde}.
    \item (F is Lipschitz) For $(M,p,z,x,t)\in Sym(n) \times \mathbb{R}^n\times \mathbb{R} \times Q_2$, $F$ has the following Lipschitz bound,
                 $$\|\nabla_{(p,z)} F\|_{L^{\infty}(Sym(n)\times\mathbb{R}^n\times \mathbb{R}\times Q_2)} \leq C$$  
        for some constant $C>0$. 
\end{enumerate}
 In particular the statement above includes equations such as, 
$$u_t - \text{tr}(A(x,t)D^2_xu(x,t))-\textbf{b}(x,t)\cdot \nabla_x u(x,t) -c(x,t)u(x,t) = 0$$
where the matrix $A(x,t)$ has eigenvalues bounded between $\lambda$ and $\Lambda$ and the functions $\textbf{b},c$ are bounded.\\

  A major breakthrough in the study of equation \eqref{main_pde} was achieved when Krylov and Safonov established a weak Harnack inequality for solutions to \eqref{main_equation} in \cite{krylov_certain_1981}. Later Tso in \cite{Tso} improved the ABP estimate used in the proof of the parabolic Harnack inequality. A comprehensive study of viscosity solutions to fully nonlinear parabolic equations was carried out in a series of landmark papers by L. Wang in \cite{LWang1}, \cite{LWang2} and \cite{LWang3}. This was done in an effort to establish parabolic analogues of results for elliptic partial differential equations developed in \cite{Caffarelli} and \cite{CKNS}. Therefore for many results concerning elliptic partial differential equations, it is natural and interesting to ask whether analogous statement holds in the parabolic setting. See for example Savin's perturbed solutions result in \cite{savin_small_2007} and its parabolic version in Y. Wang's \cite{wang_small_2012}; the partial regularity result of Silvestre, Armstrong and Smart in \cite{SSA} and its parabolic analogue by Daniel in \cite{JPD}; counterexamples to regularity of uniformly elliptic partial differential equations constructed by Nadirashvili and Vlăduţ in \cite{NADIRASHVILI2013769} and their (almost) parabolic counterparts by Silvestre in \cite{Sil_counterexample}. 
  
  In fact a $W^{2,\e}$ result in \cite{SSA} and \cite{JPD} is closely related to the dimension of the singular sets of solutions discussed in those papers. Therefore a question was raised in \cite{SSA} to find better asymptotic estimates for this $\e$. In the elliptic setting it was improved by Le in \cite{Le} and then later by Mooney in \cite{mooney_proof_2019}. As an application of our techniques in this paper, we improve the dependence of this exponent on the ellipticity ratio for the parabolic case.

  A common theme in the proof of a weak Harnack inequality, as discussed in \cite{mooney_proof_2019} for the elliptic case, is: $(A)$ Establish an ABP estimate using the area formula to gain information in measure of certain contact points, and $(B)$ Perform a delicate (Carlderon-Zygmund or Vitali) covering argument. Our method adapted from \cite{mooney_proof_2019} is global in nature and allows us to avoid a covering argument. Moreover a simple observation while trying to apply the area formula leads to another improvement.\\

We now describe our strategy of proving Theorem \ref{he}. See the discussion following it for some important consequences of weak Harnack inequality. Due to reasons mentioned in \cite[Section 6]{wang_small_2012}, it is enough to study equation \eqref{main_equation} when $C=0$. When $C \neq 0$, the statements of our results remain mostly unchanged except for the additional dependence of $\e$ in our theorems on $C$. 

By a standard ABP estimate (similar to our Lemma \ref{pgn}) the lower envelope of solutions to \eqref{main_equation} by spacetime paraboloids of openings $8^k$ (denoted by $E_{8^k}$) contact the solution in a universal fraction of $\B_1$ (See Section \ref{Preliminaries} for definitions). One then expects steeper spacetime paraboloids of opening $8^{k+1}$ to have additional contact points away from the agreement set of the solution with $E_{8^k}$. This is achieved by considering new spacetime paraboloids of opening $8^{k+1}$ tangent from below to $E_{8^k}$ at points away from its contact set, and then allowing them to flow further in time. We prove this in Lemma \ref{tching} which hinges on the fact that supersolutions to \eqref{main_equation} cannot have upward corners. Then in Lemma \ref{pgn} this process is carried out to gain new information in measure which is global in nature. For our arguments to work, we also need to ensure that all the contact points happen in the interior of the domain. This is done by using Lemma \ref{smp} and approaching the boundary from the interior in a quantitative fashion.

 The paper is organized as follows: We gather some notation and give exact statements of our results in Section \ref{Preliminaries}. Then in Section \ref{support} we establish necessary lemmata for our proof. We conclude in Section \ref{proof} by proving Theorem \ref{he}.

\section{Preliminaries}\label{Preliminaries}
In this section we fix our notations. All distances and measures are taken to be the standard $(n+1)$ dimensional Euclidean distance and Lebesgue measure. $Sym(n)$ denotes the set of all $n\times n$ real symmetric matrices.\\
 For $N \in Sym(n)$ we define the Pucci extremal operators as,
\begin{equation}\label{pucci}
    \begin{split}
        \mathcal{M}^{-}_{\lambda,\Lambda}(N) := \inf\{\text{tr}(AN)\;|\; \forall A \in &Sym(n)\text{ such that } \lambda I \leq A \leq \Lambda I \}\\
        =\lambda \left(\sum \text{positive eigenvalues of }N\right)&+\Lambda\left(\sum \text{negative eigenvalues of }N\right),\\                
    \end{split}
\end{equation} 
and
\begin{equation}
    \mathcal{M}^{+}_{\lambda,\Lambda}(N) : = -\mathcal{M}^{-}_{\lambda,\Lambda}(-N).
\end{equation}

We define the backward and forward parabolic neighbourhoods respectively as,
\begin{equation}
    \begin{split}
        &Q_{r}^B((x_0,t_0)) = \{(x,s)\in \mathbb{R}^{n}\times \mathbb{R}\;:\; |x-x_0| < r, t_0-r^2 < t\leq t_0 \},\\
        &Q_{r}^F((x_0,t_0)) = \{(x,s)\in \mathbb{R}^{n}\times \mathbb{R}\;:\; |x-x_0| < r, t_0 < t\leq t_0+r^2 \}.
    \end{split}
\end{equation}
  An arbitrary parabolic neighbourhood will be denoted as $Q$. We also recall the standard notation for a parabolic cylinder, 
  $$Q_r = Q_r^B((0,0)).$$

We further define the backward and forward spacetime parabolic balls respectively as,
  \begin{equation}
    \begin{split}
        &\B_{(x_0,t_0)} = \{(y,t)\in \mathbb{R}^{n}\times \mathbb{R}\;:\; t-t_0 < -|x-x_0|^2\},\\
        &\F_{(x_0,t_0)} = \{(y,t)\in \mathbb{R}^{n}\times \mathbb{R}\;:\; t-t_0 >|x-x_0|^2\}.
    \end{split}
  \end{equation}
  The unit backward parabolic ball is denoted as $\B_1 = \B_{(0,0)}\cap Q_1$.\\

  We now define the shapes used in deriving our estimates. For $\delta > 0, (y,s)\in \mathbb{R}^n\times\mathbb{R}$ and $\sigma\in \mathbb{R}$ consider
 \begin{equation}\label{stp}
    P^{(y,s)}_{\sigma;\delta}(x,t) = \sigma \left(t-s-\frac{|x-y|^2}{2}\right)\;\text{ for } s\leq t \leq s+\delta,\; x\in\mathbb{R}^n.
 \end{equation}  
 We will say the above function has opening and slope $|\sigma|$, vertex $(y,s)$ and flows for time $\delta$. If $\sigma>0$, we call it a concave (spacetime-)paraboloid and if $\sigma<0$ we call it a convex (spacetime-)paraboloid. An arbitrary such shape which flows for infinite time (i.e. $\delta$ is taken to be infinite) will be called $P_\sigma$, only specifying the opening.
 
 Now consider the envelope defined by concave paraboloids with a fixed opening $\sigma$ to a nonnegative continuous function $u\in C(\bar{Q})$,
 \begin{equation}
    E_\sigma(u) = \sup{\{P^{(y,s)}_{\sigma;\delta}\;\big|\;P^{(y,s)}_{\sigma;\delta} \leq u \text{ in }\bar{Q},\; (y,s)\in \mathbb{R}^{n}\times \mathbb{R}\text{ and }\delta\in \mathbb{R}_{> 0}\}}.
 \end{equation}
When there is no ambiguity for the function $u$, we will simply write $E_\sigma$. We also define the set of contact points
\begin{equation}
    A_\sigma(u) = \{(x,t)\in \mathbb{R}^{n}\times \mathbb{R}\;:\; E_\sigma(u)(x,t)=u(x,t)\}.
\end{equation}
Note that $E_\sigma$ and $A_\sigma$ have the following scaling,
\begin{equation}
    E_{\sigma}(u) = \sigma E_{1}\left(\frac{u}{\sigma}\right)\;\text{   and  }\;A_{\sigma}(u) = A_{1}\left(\frac{u}{\sigma}\right).
\end{equation}

 We now recall the notion of touching from below in the parabolic setting. Note that for the time variable, any continuity (or differentiability) is understood in terms of left continuity (or differentiability).
 \begin{definition}
 Let $\varphi$ be a smooth function and $v\in C(\overline{Q_1})$. We say $\varphi$ touches $v$ from below at $(x,t)\in Q_1$ if for some $\e>0$,
 \begin{equation}
    \varphi(\xi,\tau) \leq v(\xi,\tau),\;\;\forall (\xi,\tau)\in Q_\e^B((x,t)),\;\;\text{   and  }\;\; \varphi(x,t)= v(x,t).
 \end{equation}
 \end{definition}

 \begin{definition}
    A function $v\in C(\overline{Q_1})$ is called a viscosity supersolution if it solves $v_t-\mathcal{M}^-_{\Lambda,\lambda}(D^2_xv)\geq 0$ in the viscosity sense, i.e., for all smooth functions $\varphi$ that touches $v$ from below at some $(x,t)\in Q_1$, $\varphi_t(x,t)-\mathcal{M}^+_{\Lambda,\lambda}(D^2_x\varphi)(x,t)\geq 0$.
 \end{definition}
 
 We also recall semi-convexity in the setting of parabolic partial differential equations,
 \begin{definition}
  A function $v\in C(\overline{Q})$, is called uniformly semi-convex if it can be touched from below at every point in $Q$ by concave paraboloids of fixed opening. 
 \end{definition}
 The Lemma below gives us regularity properties of semi-convex functions,
  \begin{lemma}[Theorem 2.8, \cite{wang_small_2012}]\label{wdiff}
    If a function v is uniformly semi-convex on a set $Q_1$, then there exists a measure zero set $\mathcal{N}$ such that the following statement holds:
    For each $(x,s)\in Q_1\setminus \mathcal{N}$, $v$ has the following expansion,
    $$v(\xi,t) = p_{(x,s)}(\xi,t)+o(|\xi-x|^2+|t-s|),\;\text{ as}\;|\xi-x|\to 0,|t-s|\to 0.$$
    Here $p_{(x,s)}$ is the following function,
    $$p_{(x,s)}(\xi,t) = a+b\cdot(\xi-x)+\beta (t-s)+\frac{1}{2}(\xi-x)^tM(\xi-x),\;a,\beta\in \mathbb{R},b\in\mathbb{R}^n,M\in Sym(n).$$ 
 \end{lemma}
Since the envelopes $E_\sigma(u)$ are uniformly semi-convex (as they are touched by concave paraboloids of opening $\sigma$ from below), Lemma \ref{wdiff} ensures a.e. differentiability of $E_\sigma(u)$.\\

We now state our main result,
\begin{theorem}(Weak Harnack inequality)\label{he}
    Let $u$ be a nonnegative viscosity supersolution, $u_t-\mathcal{M}^-_{\Lambda,\lambda}(D^2_xu)\geq 0$ on $Q_2$. Then for a constant $\e=\e(n,\lambda,\Lambda)$,
    \begin{equation}
        |\{u\geq C(\Lambda,\lambda,n)u(0)t\}\cap \B_1|\leq |\B_1| t^{-\e}.
    \end{equation}
    In particular $\e$ can be taken as $\e =  \bigO(\lambda^n(n\Lambda+1)^{-n})$.
\end{theorem}

As discussed in \cite{wang_small_2012} and \cite{JPD} it suffices to assume that any supersolution is semi-concave. This is achieved by using inf convolution with spacetime paraboloids. See \cite[Lemma 4.2]{wang_small_2012} for an exact statement. This reduction ensures that the maps to which we apply the area formula are Lipschitz. Therefore we will prove Theorem \ref{he} under this assumption. 

It is now a standard consequence of the Theorem \ref{he} to obtain interior Hölder regularity for solutions to equation \eqref{main_pde},
$$\|u\|_{C^{\alpha}_{x,t}(\B_1)} \leq C(n,\lambda,\Lambda)\|u\|_{L^{\infty}(Q_2)},$$
with $\alpha\sim \exp\{c(n)(\Lambda/\lambda+\frac{1}{\lambda})^{-n}\}$ (Here the $C^\alpha_{x,t}$ norm roughly denotes the $\alpha-$Hölder norm in the space variable and $\alpha/2-$Hölder norm in the time variable. See \cite{JPD} for the exact definition of $C^{\alpha}_{x,t}(\B_1)$). 

Now defining,
$$\underline{\Theta}(u)(x,t):=\inf\{\sigma>0\;:\;P^{(y,s)}_{\sigma,t-\delta}\text{ touches u from below at }(x,t)\},$$
we can also derive the following estimate as an immediate corollary of our proof.
\begin{theorem}[$W^{2,\epsilon}$ estimate]\label{tht}
    Let $u$ be a nonnegative viscosity supersolution, $u_t-\mathcal{M}^-_{\Lambda,\lambda}(D^2_xu)\geq 0$ on $Q_2$. Then for the $\e(n,\lambda,\Lambda)$ as in Theorem \ref{he} we have,
    $$|\{\underline{\Theta}(u)\geq C'(\Lambda,\lambda,n)u(0)t\}\cap \B_1| \leq |\B_1|t^{-\e}.$$
\end{theorem}

Note that for $\delta = \epsilon/2$, solutions to equation \ref{main_pde} have an $L^{\delta}(\B_1)$ estimate for second order spatial derivatives and first order time derivatives. It was shown in \cite{JPD} that the statement of Theorems \ref{he} and \ref{tht} hold for $\epsilon \sim \lambda^{n+1}(n\Lambda+1)^{-n-2}$. We improve the dependence further to $\epsilon \sim \lambda^n(n\Lambda+1)^{-n}$. We discuss this improvement in remark \ref{improvement}.

 \section{Supporting lemmas} \label{support}
 We now derive some important lemmata that allow us to prove Theorem \ref{he}. 
\begin{lemma}\label{tching}
    Assume $u\in C(\overline{Q_1})$ solves $u_t-\mathcal{M}^{-}_{\Lambda,\lambda}(D^2u)\geq 0 $. Consider some concave paraboloid $P_{\sigma^2}$ touching $E_\sigma(u)$ from below at $(x_0,t_0)\in Q_1\setminus A_\sigma(u)$. If $P_{\sigma^2}$ flows further in time and touches $u$ in $Q_1$ at a point $(x',t')$, then $(x',t') \in A_{\sigma^2}(u)\setminus A_\sigma(u)$.

\end{lemma}

\begin{proof}
     We argue by contradiction. Assume that $(x',t')\in A_\sigma(u)$. Then there exists a concave paraboloid, $P_\sigma$, touching $E_\sigma(u)$ (and hence $u$) from below at $(x',t')$. Also assume $P_0$ to be a concave paraboloid of opening $\sigma$ touching $E_\sigma(u)$ from below at $(x_0,t_0)\in Q_1\setminus A_\sigma(u)$.

     By definition of $E_\sigma$, $E_\sigma(x_0,t) \leq P_0(x_0,t)$ for all time $t \geq t_0$. Since the slope of $P_{\sigma^2}$ is greater than $P_0$, therefore for $t>t_0$, $E_\sigma(x_0,t) < P_{\sigma^2}(x_0,t)$. Thus $P_\sigma(x_0,t')<P_{\sigma^2}(x_0,t')$. However we also have $P_{\sigma^2}(x',t') = P_\sigma(x',t')$. Since $P_{\sigma^2}$ has a bigger opening than $P_\sigma$, at time $t=t'$ $P_{\sigma^2}$ and $P_{\sigma}$ cannot be tangent to each other at $(x',t')$.
      
     Now consider $P = \max\{P_\sigma,P_{\sigma^2}\}$ at $t=t'$. This shape sits below $u$ and touches $u$ at $(x',t')$ and has an upward corner at $(x',t')$. Now by rotating the spatial coordinates, translating by a linear function in the space variable and rescaling appropriately, we can transform the shape $P$ at $t=t'$ to $B = \max\{1-|x|^2,0\}$, such that it touches $u$ at $x=(1,0,\dots,0)$. While performing these changes our equation changes to $u_t-\mathcal{M}^{-}_{\lambda,\Lambda}(D^2u)\geq -n\lambda \sigma $. Now using the barrier as in \cite{mooney_proof_2019}, $\psi = 1-|x|+A(|x|-1)^2$, we can construct $g(x,t) = \psi(x)+M(t-s)$ for large $M$ and $s$ such that $t'-s$ is small and $g$ sits below $u$. However it cannot touch $u$ from below since $g_t-\mathcal{M}^-_{\Lambda,\lambda}(D^2_xg) = M-2\lambda A + (n-1)\Lambda < -n\lambda \sigma $ for large $A$. 
\end{proof}

The following lemma exploits semiconvexity of the envelope which allows us to obtain information (in measure) about their vertices from their contact points. 

\begin{lemma}\label{tf}
    For a continuous positive function $u\in C(\overline{Q_1})$, consider the envelope $E_\sigma$. For a set of contact points $\T\subseteq Q_1$ to $E_\sigma$ by concave paraboloids, $P_{\sigma^2}$, the vertices of these shapes, $V$, will have measure atleast a fraction of $\T$, i.e.,
    $$|V| \geq \left(1-\frac{1}{\sigma}\right)^{n+1}|\T|.$$

\end{lemma}

\begin{proof}
    We have the following information at the contact points $ (x,t)\in \T\setminus \mathcal{N}$ (here $\mathcal{N}$ is a measure zero set where $E_\sigma$ is not differentiable),
    \begin{equation}
        E_\sigma(x,t) = \sigma^2(t-s-\frac{|y-x|^2}{2}),
            \end{equation}
            \begin{equation}
                \nabla_x E_\sigma(x,t) = \sigma^2(y-x),
            \end{equation}
        \begin{equation}\label{tuu}
            \partial_t E_\sigma(x,t) \leq \sigma,
        \end{equation}
        \begin{equation}\label{dll}
            D^2E_\sigma(x,t) \geq -\sigma I_{n\times n}.
        \end{equation}
         The first two equations come from the fact that a concave paraboloid $P_{\sigma^2}$ touches $E_\sigma$ from below at $(x,t)$ and the last two come from the fact that some concave paraboloid $P_\sigma$ also touches $E_\sigma$ from below by definition. Thus the contact point to vertex mapping $L: \mathcal{T}\setminus \mathcal{N}\to V$ is given by,
        $$L: (x,t) \mapsto (y = x+\frac{\nabla_x E_\sigma(x,t)}{\sigma^2}, s = t-\frac{|\nabla_x E_\sigma(x,t)|^2}{2\sigma^4}-\frac{E_\sigma(x,t)}{\sigma^2}). $$
        We assume for now that $E_\sigma$ can be touched from above by convex paraboloids of opening $C$ on $\T\setminus\mathcal{N}$. Therefore the map $L$ is Lipschitz since $E_\sigma$ can be uniformly touched from above and below by paraboloids of fixed openings. 
        
        Moreover this map must be injective. We show this by contradiction. Consider a concave paraboloid $P_{\sigma^2}$ and assume it touches $E_\sigma$ at two points $(x_1,t_1),(x_2,t_2)$. Let $P_1,P_2$ be two concave paraboloids of opening $\sigma$ that touch $E_\sigma$ from below at these two points respectively. We assume $t_2\geq t_1$. Since the opening of $P_{\sigma^2}$ is strictly greater that $P_2$, if $P_{\sigma^2}$ touches $P_2$ at some point $(x_2,t_2)$, then $P_{\sigma^2} < P_2$ everywhere except $(x_2,t_2)$. However by definition of $E_\sigma$, $E_\sigma\geq P_2$. Since $t_1\leq t_2$, we now have, 
        $$E_\sigma(x_1,t_1)\geq P_2(x_1,t_1)> P_{\sigma^2}(x_1,t_1) = E_{\sigma}(x_1,t_1),$$
         a contradiction.
        \\

        We further restrict ourselves to $T\setminus \mathcal{N}'$, where $\mathcal{N}'$ a measure zero set is defined as,
         $$\mathcal{N}' = \mathcal{N}\cup \{\text{The map L is not differentiable}\}.$$
         Now to calculate the Jacobian $D_{(x,t)}((y,s))$, we do a calculation similar to \cite{wang_small_2012}. Its first $n$ rows can be written as,
          $$\left(I_{n\times n}+\frac{D^2_xE_\sigma}{\sigma^2},\frac{\nabla_x \partial_t E_\sigma}{\sigma^2}\right).$$
           Its last row can be written as, 
           $$\left(0,1-\frac{\partial_tE_\sigma}{\sigma^2}\right)-\left(I+\frac{D^2_xE_\sigma}{\sigma^2}, \frac{\nabla_x\partial_tE_\sigma}{\sigma^2}\right)\boldsymbol{\cdot} \frac{\nabla_xE_\sigma}{\sigma^2}.$$
            Thus to check positivity and also calculate the determinant, we can eliminate the second term above. By the Area formula, inequalities \eqref{tuu},\eqref{dll} and using the fact that $L$ is bijective,
        \begin{equation}\label{af}
        |V| \geq \int_{\T\setminus \mathcal{N}'}\left(1-\frac{\partial_t E_\sigma}{\sigma^2}\right)\det\left(I+\frac{D^2_xE_\sigma}{\sigma^2}\right)\geq (1-\frac{1}{\sigma})^{n+1}|\T|
        \end{equation}
       Now define,
        $$\T_N = \{(x,t)\in \T\setminus\mathcal{N}'|\; E_\sigma \text{ is touched from above by convex paraboloids of opening } N\}.$$
         From equation \eqref{af}, $|V| \geq (1-1/\sigma)^{n+1}|\T_N|$ for every $N$. Moreover, since $E_\sigma$ is uniformly semi-convex we have $\T\setminus\mathcal{N}' = \cup_{N\in\mathbb{N}}\T_N$. 

\end{proof}

The following lemma is our version of the ABP estimate which gives information (in measure) about the additional contact points we get in Lemma \ref{tching}.

\begin{lemma}\label{pgn}
    Let $u\in C(\overline{Q_1})$ be a uniformly semi-concave function such that $u_t-\mathcal{M}^-_{\Lambda,\lambda}(D^2u)\geq 0$. Let $T\subseteq Q_0\setminus A_{\sigma}(u)$ be points where shapes $P_{\sigma^2}$ touch the envelope $E_\sigma$ from below. Then letting these shapes flow and assuming they touch $u$ in $Q_1$,
    $$|A_{\sigma^2}(u)\setminus A_{\sigma}(u)|\geq \frac{\lambda^{n-1}\min\{1,\lambda\}\left(1-\frac{1}{\sigma}\right)^{n+1}}{(1+n\Lambda)^{n}}|T|.$$
\end{lemma}
\begin{proof}
    From Lemma \ref{tf} the vertices $V$ have measure, $|V|\geq (1-1/\sigma)^{n+1} |T|$. From Lemma \ref{tching} we know the paraboloids $P_{\sigma^2}$ will touch at $A_{\sigma^2}(u)\setminus A_\sigma(u)$. Let $(x,t)\in (A_{\sigma^2}(u)\setminus A_\sigma(u))\setminus \mathcal{N}$ ($\mathcal{N}$ is the set of non differentiability of u) be a new point of contact. Then we have the following information at $(x,t)$ for some vertex $(y,s)\in V$,
   \begin{equation}
        u(x,t) = \sigma^2(t-s-\frac{|y-x|^2}{2}),
            \end{equation}
            \begin{equation}
                \nabla_x u = \sigma^2(y-x),
            \end{equation}
        \begin{equation}\label{tu}
            \partial_t u \leq \sigma^2,
        \end{equation}
        \begin{equation}\label{dl}
            D^2_xu \geq -\sigma^2 I_{n\times n}.
        \end{equation}
        From inequality \eqref{tu}, $\mathcal{M}^{-}_{\Lambda,\lambda}(D^2u) \leq \sigma^2$. By inequality \eqref{dl} we also know the negative eigenvalues of $D^2_xu$ have to be atleast $-\sigma^2$. Therefore, 
        \begin{equation}
        D^2_xu \leq \sigma^2\frac{(\Lambda(n-1)+1)}{\lambda} I_{n\times n}.
        \end{equation}
        Moreover due to inequality \eqref{dl} and the equation, $u_t \geq -n\Lambda\sigma^2$. The contact point to vertex map is now given by,
        $$(x,t) \mapsto (y = x+\frac{\nabla_x u}{\sigma^2}, s = t-\frac{|\nabla_x u|^2}{2\sigma^4}-\frac{u}{\sigma^2}).$$
        With an argument similar to Lemma \ref{tf}, this map is Lipschitz and we can apply the area formula to get,

        \begin{equation}
            |V| \leq \int_{(A_{\sigma^2(u)}\setminus A_\sigma(u))\setminus\mathcal{N}}\left(1-\frac{\partial_t u}{\sigma^2}\right)\det\left(I+\frac{D^2_xu}{\sigma^2}\right).
        \end{equation} 
        Now we can breakup the estimate into two cases,
        \begin{itemize}
            \item Case 1: Atleast one eigenvalue of $D^2_xu$ is negative. Thus the bound becomes,
             $$\left(1-\frac{\partial_t u}{\sigma^2}\right)\det\left(I+\frac{D^2_xu}{\sigma^2}\right) \leq (1+n\Lambda)\left(1+\frac{\Lambda(n-1)+1}{\lambda}\right)^{n-1}. $$
             \item Case 2: All eigenvalues of $D^2_xu$ are positive. Since $u_t-\mathcal{M}^-_{\Lambda,\lambda}\geq 0$, we get $u_t\geq 0$ and hence,
             $$\left(1-\frac{\partial_t u}{\sigma^2}\right)\det\left(I+\frac{D^2_xu}{\sigma^2}\right) \leq \left(1+\frac{\Lambda(n-1)+1}{\lambda}\right)^n.$$  
        \end{itemize}
        Thus in both cases since $\lambda\leq \Lambda$ we get our desired bound.
    \end{proof}

\begin{rmk}\label{improvement}
    We are able reduce the exponent of $(n\Lambda+1)^{-1}$ by $1$ when compared to a similar ABP estimate in \cite{wang_small_2012} and \cite{JPD}. By a carefully designed iteration argument in our proof in Section \ref{proof} which decouples the ABP estimate from the localization part of the proof, we further reduce the exponent of $(\lambda(n\Lambda+1))^{-1}$ by another factor.
\end{rmk}
Now we consider an important barrier,
\begin{equation}
\phi_{(x_0,t_0)}(x,t) = (t-t_0)^{-\alpha}(e^{-\beta|x-x_0|^2/(t-t_0)}-e^{-\beta}) \text{ for } t>t_0.
\end{equation}
For constants $\alpha$ and $\beta$ given by
\begin{equation}\label{ab}
\beta = 4\left(1+\frac{(n-1)\Lambda+1}{\lambda}\right) \text{ and } \alpha = 2n\beta \Lambda+1,
\end{equation}
the above function is a subsolution, $\phi_t-\mathcal{M}^{+}_{\lambda,\Lambda}(D^2\phi) \leq 0$, on $\mathcal{F}_{(x_0,t_0)}$. The following lemma is a quantitative version of the strong maximum principle,

\begin{lemma}\label{smp}
    Let $u\in C(\overline{Q_2})$ be a positive function that such that $u_t-\mathcal{M}^-_{\Lambda,\lambda}(D^2u) \geq 0$. For a spatial neighbourhood $L = \{|x-x_0|\leq \epsilon\}\times\{t_0\}\subseteq \B_1$,
    $$\inf_{x\in L}\{u(x)\} \leq e^{\beta}(1+\frac{1-t_0}{\e^2})^{\alpha+1}u(0,1).$$
\end{lemma}
\begin{proof}
    Consider the vertex $(x_1,t_1) = (x_0,t_0-\e^2)$ and let $\phi_1 = \phi_{(x_1,t_1)}$. On the open set $Q_{\sqrt{1-t_0}}^F((x_0,t_0))\cap \mathcal{F}_{(x_1,t_1)}$, compare $u$ with the function $\varphi = 2\phi_1u(0,1)/\phi_1(0,1)$. If at time $t=t_0$ $\varphi\leq u$, then by the comparison principle, $\varphi\leq u$ at all future times. However $u(0,1) < \varphi(0,1)$, therefore at time $t_0$ there exists $\tilde{x}$ such that $u(\tilde{x},t_0)\leq \varphi(\tilde{x},t_0)$. Note that this point has to lie in $L$, i.e., $(\tilde{x},t_0)\in L$, since at $t=t_0$ $\phi_1\geq 0$ only on $L$. By a computation we get,
    $$u(x,t_0) \leq 2u(0,1)\frac{\epsilon^{-2\alpha}(e^{\frac{-\beta|x-x_0^2|}{\e^2}}-e^{-\beta})}{(1+\e^2-t_0)^{-\alpha}(e^{\frac{-\beta|x_0|^2}{(1+\e^2-t_0)}}-e^{-\beta})}\leq (1+\frac{1-t_0}{\e^2})^{\alpha+1}\frac{(1-e^{-\beta})}{e^{-\beta}\beta}u(0,1).$$
    Since $L\subseteq \mathcal{B}_1$, $(|x_0|+\e)^2 \leq -(t_0-1)$, thus $|x_0|^2\leq 1-t_0-\e^2$. Therefore to bound the denominator (let $\frac{1-t_0}{\e^2} = c$),
    $$e^{\frac{-\beta|x_0|^2}{(1+\e^2-t_0)}}-e^{-\beta} \geq e^{-\beta\frac{c-1}{c+1}}-e^{-\beta} = e^{-\beta}(e^{\frac{2\beta}{c+1}}-1)\geq e^{-\beta}\frac{2\beta}{c+1},$$
    
\end{proof}

\section{Main Proof}\label{proof}
We now prove Theorem \ref{he} using the strategy outlined in Section \ref{intro}.

\begin{proof}[\textbf{Proof of Theorem \ref{he}}]
For $M = \alpha+1$ ($\alpha$ as in \eqref{ab}) and $\e_k = 2^{-1}4^{-\frac{k+1}{M+2}}$, consider the sets
$$\mathcal{B}_{\e_k} = \mathcal{B}_{(0,-\e_k)}\cap Q_{1-\e_k}.$$
 Also let $\sigma = 8$ and $ C_0 = e^\beta$ ($\beta$ as in \eqref{ab}). To simplify notation we also denote $c_\sigma$ to be the constant in Lemma \ref{tf} and $c_\sigma'$ to be the constant in Lemma \ref{pgn}. We normalize $u$ by dividing it by $64\cdot16^MC_0u(0,1)$.
    For $\e>0$ satisfying $\sigma^{-\e}\geq \max\{1-\frac{c_\sigma'}{2},4^{-\frac{1}{M+2}}\}$, we show the following decay in measure,
    \begin{equation}\label{main_decay}
    |\mathcal{B}_1 \setminus A_{\sigma^k}| \leq  \sigma^{-\e k}|\B_1| . 
    \end{equation}
    \textit{Boundedness of $u$ on $A_\sigma$}: Since $u(0,1) = (64\cdot16^MC_0)^{-1}$, for points $p\in A_{\sigma^k}\cap \mathcal{B}_{\e_k}$ we claim that $u(p) \leq \sigma^k$. This observation along with \eqref{main_decay} proves the weak Harnack inequality.
    We proceed by contradiction. If $u(p)> \sigma^{k}$, since a paraboloid of opening $\sigma^{k}$ touches $u$ from below at $p$, then on the set $L_1 = \{|x-a|\leq \epsilon_k/2\}\times\{b\}\subseteq \B_1$ for some $(a,b)\in \B_1$, 
    $$\inf_{(x,t)\in L_1}\{u(x,t)\}\geq \frac{\sigma^k}{2}.$$
     However by Lemma \ref{smp} the infimum can at most be $64^{-1}8^{-M}\e_k^{-2M}$, a contradiction.\\

     The decay \eqref{main_decay} is now proved by induction. To ensure the new contact points gained after applying Lemma \ref{pgn} lie in $\B_1$, we only use the information available in $\B_{\e_k}$. To that end we show, 
     \begin{equation}\label{il}
        |(A_{\sigma^{k+1}}\setminus A_{\sigma^k})\cap \B_{\e_{k+1}}|\geq c_\sigma' |\mathcal{B}_{\e_k}\setminus A_{\sigma^k}|.
     \end{equation}
    \textit{Vertex of new paraboloids $P_{\sigma^{k+1}}$}: For any point $p\in \B_{\e_k}$, $Q_p^F(\delta_{k+1}) \subseteq \B_{\e_{k+1}}$. Here $\delta_{k+1} = \e_{k+1}/4$. Let $P_0 = P_{\sigma^{k+1}}$ be a paraboloid that touches $E_{\sigma^k}$ from below at $p_0 = (x_0,t_0)\in \B_{e_k}$. If its vertex was $(y_0,s_0)$, then $y_0 \in  B_{\delta_{k+1}/4}(x_0)$. If not, then in a spatial neighbourhood
      $$L_2 = \{|x-\tilde{x}_0|\leq \frac{\delta_{k+1}}{8}\}\times\{t_0\},$$ 
      around the point $\tilde{x}_0=x_0+\frac{\delta_{k+1} (y_0-x_0)}{4|y_0-x_0|}$ the paraboloid $P_0$ is greater than $ 3\sigma^{k+1}\delta_{k+1}^2/64$. However by our choice of $\sigma$, 
     $$\inf_{(x,t_0)\in L_2}\{u(x,t_0)\} \geq \inf_{(x,t_0)\in L_0}\{P_0\} \geq \frac{3\sigma^{k+1}\delta_{k+1}^2}{64} \geq 64^{-1}2^{-M}\delta_{k+1}^{-M},$$ 
     a contradiction to Lemma \ref{smp}.\\
      Moreover around $y_0$, on a spatial neighbourhood $L_0 = \{|x-y_0| \leq \delta_{k+1}/8\}\times \{t_0\}$,  
      $$\inf_{(x,t_0)\in L_0}\{u(x,t_0)\} \geq \inf_{(x,t_0)\in L_0}\{P_0\} \geq P_0(y_0,t_0) - \sigma^{k+1}\frac{\delta_{k+1}^2}{64}.$$ 
      Therefore we can similarly argue $t_0-s_0 \leq \delta_{k+1}^2/16$. Since otherwise 
      $$P_0(y_0,t_0) \geq P_0(y_0,s_0)+\sigma^{k+1}(t_0-s_0) \geq \sigma^{k+1}\frac{\delta_{k+1}}{16}.$$ 
      Hence $(y_0,s_0)\in Q_{p_0}^B(\delta_{k+1}/4)$.\\
     \textit{Contact at interior points}: Now for any point $p\in Q_{p_0}^B(\delta_{k+1}/4)$, $Q_{p}^F(\delta_{k+1}/2) \subseteq \mathcal{B}_{\e_{k+1}}$. Then starting at $p = (p_x,p_t)\in Q_{p_0}^B(\delta_{k+1}/4)$,  if we let a paraboloid $P_{\sigma^{k+1}}$ tangent from below to $E_{\sigma^{k}}$ flow, it must touch $u$ in $Q_p^F(\delta_{k+1}/2)$. It follows by contradiction as before. Let $(y,s)$ be the vertex of $P_{\sigma^{k+1}}$. Then on the set $L = \{|x-y| \leq \delta_{k+1}/2\}\times \{p_t+\delta_{k+1}^2/4\}$, 
     $$u \geq P_0\geq 3\sigma^{k+1}\frac{\delta_{k+1}}{16},$$ 
     which contradicts Lemma \ref{smp}. Moreover the new contact point will be in 
     $$\{P_0\geq 0\}\subseteq Q_p^F(\frac{\delta_{k+1}}{2})\subseteq \B_{\e_{k+1}}.$$
      Inequality \eqref{il} follows by applying Lemma \ref{pgn} to $u/\sigma^{k-1}$.
    \\ 
    
    \textit{Induction argument}: We now proceed by induction. The base case $k=0$ is trivial. For the induction step there are two cases to consider,
    
     \begin{itemize}
        \item If $|\B_{\e_k}\setminus A_{\sigma^k}|\leq \frac{|\B_{1}\setminus A_{\sigma^k}|}{2}$, then, 
        $$|\B_{1}\setminus A_{\sigma^{k+1}}| \leq |\B_{1}\setminus A_{\sigma^k}| \leq 2 |\B_1\setminus \B_{\e_k}|\leq 2\e_k|\B_1| \leq 4^{-\frac{k+1}{M+2}} |\B_1|.$$
        
        \item  If $|\B_{\e_k}\setminus A_{\sigma^k}| \geq \frac{|\B_{1}\setminus A_{\sigma^k}|}{2}$, from inequality \eqref{il}, 
        $$|(A_{\sigma^{k+1}}\setminus A_{\sigma^{k}})\cap \B_{1}|\geq \frac{c_\sigma'}{2}|\B_{1}\setminus A_{\sigma^k}|.$$
        Thus,
        $$|\B_{1}\cap A_{\sigma^{k+1}}| \geq |\B_{1}|- |\B_{1}\setminus A_{\sigma^k}|+|A_{\sigma^{k+1}}\cap (\B_{1}\setminus A_{\sigma^k})| \geq |\B_{1}|- (1-\frac{c_\sigma'}{2})|\B_{1}\setminus A_{\sigma^k}|$$
        $$\implies |\B_{1}\setminus A_{\sigma^{k+1}}| \leq \left(1-\frac{c_\sigma'}{2}\right)|\B_{1}\setminus A_{\sigma^k}|.$$
        
     \end{itemize}
     Therefore \eqref{main_decay} follows.\\

      \textit{Asymptotic analysis of $\e$}: To ensure $1-c_{\sigma}'/2 \leq \sigma^{-\e}$, we must have $\e \leq c_\sigma'/2\ln(\sigma)$ since $\ln(1-c_\sigma'/2) \leq -c_\sigma'/2$. Therefore, 
      $$\e \leq \left(2^{n+2}\ln(4)\left(\frac{1+n\Lambda}{\lambda}\right)^n\right)^{-1}.$$
       Furthermore we also need to ensure $\sigma^\e \leq 4^{\frac{1}{M+2}}$. Thus 
       $$\e \leq (M+2)^{-1}.$$
     Hence $\e =\bigO( \lambda^{n}(n\Lambda+1)^{-n})$. 

\end{proof}
\begin{rmk}\label{tr}
As a corollary of the preceding proof, we also obtain Theorem \ref{tht}. Since, by definition, $u$ is touched from below by paraboloids of opening $\sigma$ on the sets $A_\sigma$. 
\end{rmk}

\section{Acknowledgment}
The author would like to thank Professor Connor Mooney for asking a question that motivated this research. The author was also supported by Professor C. Mooney's NSF grant DMS-2143668.

\end{document}